\documentclass[12pt]{article}
\usepackage{epsfig,psfrag,amsmath,amssymb,latexsym}
\usepackage{amscd }
\usepackage{amsfonts}
\usepackage{graphicx}
\usepackage{wasysym}
\usepackage{color}
\newtheorem{fact}{Fact}[section]

\newtheorem{theorem}[fact]{Theorem}

\newtheorem{corollary}[fact]{Corollary}

\newtheorem{definition}[fact]{Definition}

\newtheorem{lemma}[fact]{Lemma}

\newtheorem{remark}[fact]{Remark}

\newenvironment{proof}[1][Proof]{\textbf{#1.} }{\ \rule{0.5em}{0.5em}}

\numberwithin{equation}{section}

\newcommand{\weg}{{\hexagon}}
\newcommand{\n}{{\mathcal N}}
\newcommand{\im}{{\rm Im}}

\newcommand{\pres}{{\rm pres}}
\newcommand{\sym}{{\rm sym}}
\newcommand{\id}{{\rm Id}}
\newcommand{\dist}{{\rm dist}}

\newcommand{\conv}{\operatorname{conv\, hull}}
\newcommand{\defects}{\operatorname{defects}}
\newcommand{\SO}{\operatorname{SO}}
\newcommand{\OmegaRepr}{\Omega^*}

\def\R{{\mathbb R}}  
\def\N{{\mathbb N}}  
\def\Z{{\mathbb Z}}  
\def\C{{\mathbb C}}  

\def\T{{\mathcal{T}}}

\newcommand{\cdrei}{c_{2}}
\newcommand{\cvier}{c_{3}}
\newcommand{\cfuenf}{c_{4}}
\newcommand{\csechs}{c_{5}}
\newcommand{\csieben}{c_{6}}
\newcommand{\cacht}{c_7}
\newcommand{\cneun}{m_1}
\newcommand{\czehn}{c_{8}}
\newcommand{\celf}{{c_{1}}}
\newcommand{\czwoelf}{c_{10}}
\newcommand{\cdreizehn}{c_{11}}
\newcommand{\cvierzehn}{c_{12}}
\newcommand{\cfuenfzehn}{c_{13}}
\newcommand{\csechzehn}{c_{14}}
\newcommand{\csiebzehn}{c_{9}}

\begin{document}
\thispagestyle{empty}

\begin{center}
{\LARGE Spontaneous breaking of rotational symmetry \\
in the presence of defects}\\[3mm]
{\large Markus Heydenreich\footnote{Mathematisch Instituut, Universiteit Leiden, P.O.~Box 9512, 2300~RA Leiden, The Netherlands; Centrum Wiskunde \& Informatica (CWI), P.O.~Box 94070, 1090~GB Amsterdam, The Netherlands; email: markus@math.leidenuniv.nl} \hspace{1cm} 
Franz Merkl\footnote{Mathematical Institute, University of Munich,
Theresienstr.\ 39,
D-80333 Munich,
Germany.
e-mail: merkl@mathematik.uni-muenchen.de
}
\hspace{1cm} 
Silke W.W.\ Rolles\footnote{Zentrum Mathematik, Bereich M5,
Technische Universit{\"{a}}t M{\"{u}}nchen,
D-85747 Garching bei M{\"{u}}nchen,
Germany.
e-mail: srolles@ma.tum.de}
\\[3mm]
{\small \today}}\\[3mm]
\end{center}

\begin{abstract}
We prove a strong form of spontaneous breaking of rotational symmetry
for a simple model of two-dimensional crystals with random defects 
in thermal equilibrium  at low temperature. 
The defects consist of isolated missing atoms. 
\end{abstract}




\section{Introduction}

\subsection{Motivation}

Solid state physics is about crystals. In spite of the 
tremendous achievements and numerous applications of solid state physics,
existence of crystals is mathematically not rigorously understood.
In particular, understanding the melting transition from crystals to liquids 
seems out of reach for mathematicians. 

One gets a better understanding of crystallization phenomena (classical 
and liquid crystals) by studying spontaneous 
breaking or preservation of spatial symmetries like translations and 
rotations. Preservation of translational symmetry is well understood in 
two dimensions, see for example \cite{Richthammer2007}. Hardly any 
mathematical results in realistic models are known in three dimensions. 
Among the more recent results on translational symmetry breaking in 
crystalline systems, 
we mention Aizenman, Jansen, and Jung \cite{AizenmanJansenJung2010}.

By the work of Theil \cite{Theil06}, crystallization at temperature zero in 
two dimensions is much better understood than crystals at positive 
temperature. On a macroscopic scale, geometric rigidity is well understood.
This starts with a result of Liouville. Consider 
a continuously differentiable map such that the derivative at any point is 
a rotation. By Liouville's result it is indeed globally a rotation. 
Among the more recent advances, we highlight Friesecke, James, and M{\"u}ller 
\cite{friesecke-james-mueller2002} proving a geometric rigidity result. 

In \cite{MerklRolles2009}, spontaneous breaking of rotational 
symmetry was shown for a toy model of a crystal without defects. 
However, crystals at positive temperature exhibit defects. These can
be all kinds of local defects (e.g.\ missing atoms) and various non-local 
defects. In this work, we consider a variant of the model from 
\cite{MerklRolles2009} which allows the simplest type of local defects, 
isolated missing single atoms. 
Our approach can be generalized in a straightforward way to isolated islands 
of missing atoms as long as the islands are of bounded size. 
The model forbids non-local defects like crystal boundaries and
dislocation lines by definition. Furthermore, to make the presentation as 
simple as possible, we work in two dimensions although this is not 
essential. We see the current work as one step towards
a better mathematical understanding of rotational symmetry breaking in 
crystals. 
The presence of defects makes a Fourier analysis technique inappropriate 
for our model. It is replaced by the geometric rigidity result from 
\cite{friesecke-james-mueller2002}, which therefore  
is an important ingredient in our analysis.

\subsection{The model}
\label{section1}

\paragraph{Assumptions.}
Throughout, we fix 
\begin{enumerate}
\item a real-valued potential function $V$ defined in an open 
interval containing $1$. We assume that $V$ is twice continuously 
differentiable with $V''>0$ and $V'(1)=0$. 
\item $\alpha>0$ sufficiently small, depending on $V$.
(More specifically, $\alpha$ needs to be so small that $V$ is 
defined on $[1-\alpha,1+\alpha]$ and Corollary 
\ref{cor-single-triangle} below holds.)
\item $l\in(1-\alpha/2,1+\alpha/2)$. This parameter equals the distance of 
neighboring particles in the standard configuration defined in \eqref{omega-l}
below. Thus, it is a control parameter for the ``pressure'' of the system. 
\end{enumerate}

Let $(A_2,E)$ denote the triangular lattice, viewed as an undirected graph: 
$A_2=\Z+\tau\Z$ with $\tau=e^{\pi i/3}$ and 
$E=\{\{x,y\}:x,y\in A_2,|x-y|=1\}$; here $|z|$ denotes
the Euclidean length of $z\in\C$. We write $x\sim y$ if 
$\{x,y\}\in E$. 

Let $N\in\N$. We define the set $\OmegaRepr_{l,N}$ of configurations 
$\omega$ with periodic boundary conditions to consist of all 
$\omega\in(\C\cup\{\weg\})^{A_2}$ such that 
\begin{align}
\label{def-OmegaRepr}
\omega(x+Nz)=\omega(x)+lNz \text{ for all }x,z\in A_2
\text{ with }\omega(x)\neq\weg,
\end{align}
and $\omega(x+Nz)=\omega(x)$ for $x,z\in A_2$ with $\omega(x)=\weg$.
For $x\in A_2$, $\omega(x)\in\C$ is interpreted as the location of the 
particle with index $x$. If $\omega(x)=\weg$, then there is a 
{\it hole} or a {\it defect} associated with $x$. 
Note that any $\omega\in\OmegaRepr_{l,N}$ is uniquely determined by its restriction to the set of representatives 
\begin{align}
& I_N:=
\{x+\tau^2 y:x,y\in\{0,\ldots,N-1\}\}
\end{align}
of $A_2/NA_2$.  
This allows us to identify $\OmegaRepr_{l,N}$ with $(\C\cup\{\weg\})^{I_N}$. 
Furthermore, two configurations $\omega,\omega'\in\OmegaRepr_{l,N}$ are identified if there exists $z\in A_2$ such that for all $x\in A_2$ one has $\omega(x)=\omega'(x+Nz)$. 
Let $\underline\Omega_{l,N}$ be the quotient space with respect to the 
equivalence relation given by this identification. 
One may identify $\underline\Omega_{l,N}$ with a measurable set of 
representatives $\underline\Omega_{l,N}\subset\OmegaRepr_{l,N}$.

We introduce the set 
\begin{align}
& \Lambda_{lN}:=[0,lN)+\tau^2 [0,lN)
\end{align}
of representatives for $\C/lNA_2$.
Although the precise choice of the set of representatives for 
$\underline\Omega_{l,N}$ in $\OmegaRepr_{l,N}$ is 
irrelevant, a possible choice is $\omega(x)\in\Lambda_{lN}$ for the 
lexicographically smallest $x\in I_N$ with $\omega(x)\neq\weg$ 
if $\omega$ is not the constant configuration with value $\weg$.

Let 
\begin{align}
\defects(\omega):=\omega^{-1}(\{\weg\})\cap I_N
\end{align}
denote the set of defects in the configuration $\omega$. 
For $x\in I_N$ and $z\in\{1,\tau\}$, let 
\begin{align}
\Delta_{x,z}:=\{x+sz+t\tau z:s,t>0,s+t<1\} 
\end{align}
denote the open triangle with corner points $x$, $x+z$, and $x+\tau z$. 
Let 
\begin{align}
\T_N:=\{\Delta_{x,z}:x\in I_N,z\in\{1,\tau\}\}
\quad\text{and}\quad
\T:=\{\Delta_{x,z}:x\in A_2,z\in\{1,\tau\}\}.
\end{align}
Note that the closures of the triangles in $\T_N$ cover $\Lambda_{1N}$.
Let 
\begin{align}
\n:=\{\tau^j:j\in\Z\}
\end{align}
denote the set of neighbors of $0$ in $A_2$.

The space $\Omega_{l,N}$ of {\it allowed} configurations consists of all 
$\omega\in\underline\Omega_{l,N}$ satisfying the following properties
($\Omega 1$)--($\Omega 4$): 
\begin{enumerate}
\item[($\Omega 1$)] $|\omega(x)-\omega(y)|\in(1-\alpha,1+\alpha)$ 
for all $x,y\in A_2$ with $x\sim y$, 
$\omega(x)\neq\weg$, and $\omega(y)\neq\weg$. 
\item[($\Omega 2$)] Defects are isolated in the following sense.
For all $x,y\in A_2$, one can have 
$\omega(x)=\weg$ and $\omega(y)=\weg$ only if $x=y$ or $|x-y|>2$. 
This means that nearest and next-nearest neighbors of defects 
are present. 
\end{enumerate}
For $x\in A_2$, let 
\begin{align}
\hat\omega(x)=\left\{\begin{array}{ll}
\omega(x) & \text{ if }\omega(x)\neq\weg, \\
\frac16\sum_{z\in\n}\omega(x+z) & \text{ if }\omega(x)=\weg. 
\end{array}\right. 
\end{align}
Extend $\hat\omega$ piecewice affine linearly to a map 
$\hat\omega\colon\C\to\C$ requiring that $\hat\omega$
is affine linear on the closure of every triangle in $\T$.
We require:
\begin{enumerate}
\item[($\Omega 3$)] $\hat\omega\colon\C\to\C$ is bijective.
\item[($\Omega 4$)] For all $x\in A_2$ and all $z\in\n$, one has 
\begin{align}
\im\left(
\frac{\hat\omega(x+\tau z)-\hat\omega(x)}{\hat\omega(x+z)-\hat\omega(x)}
\right)>0.
\end{align} 
In other words, $\hat\omega$ preserves orientations. 
\end{enumerate}
We remark that we could drop condition ($\Omega 4$) because it follows
from the other conditions ($\Omega 1$)--($\Omega 3$). Since the proof 
of this fact 
is more analytic than stochastic and is not needed in the current paper, 
we skip it. 

Note that the {\it standard configuration} 
\begin{align}
\label{omega-l}
\omega_l\colon A_2\to\C,v\mapsto lv 
\end{align}
is an allowed configuration. 
Thus, $\Omega_{l,N}\neq\emptyset$. 

Let $m\in\R$; $m$ has the interpretation of a {\it chemical potential}. 
It parametrizes the energetic costs of a defect. 
Define the Hamiltonian 
\begin{align}
H_{m,N}(\omega):=\frac12\sum_{x\in I_N}\sum_{y\in A_2:y\sim x}
1_{\{\omega(x)\neq\weg,\omega(y)\neq\weg\}}V(|\omega(x)-\omega(y)|)
+m\sum_{x\in I_N}1_{\{\omega(x)=\weg\}}
\end{align}
for $\omega\in\Omega_{l,N}$. 

Let $\lambda$ denote the Lebesgue measure on $\C$. 
We endow $(\C\cup\{\weg\})^{I_N}$ with the reference measure 
$(\lambda+\delta_\weg)^{I_N}$. 
This yields a reference measure on $\OmegaRepr_{l,N}$. 
Restricting this reference measure to $\underline\Omega_{l,N}$ and using 
the above identification, 
this defines in turn a reference measure $\mu_N$ on $\underline\Omega_{l,N}$. 
Note that $\mu_N(\Omega_{l,N})<\infty$ as a consequence of $(\Omega 1)$.

For $\beta>0$, let 
\begin{align}
P_{\beta,m,N}(d\omega)
:=\frac{1}{Z_{\beta,m,N}}e^{-\beta H_{m,N}(\omega)}\,\mu_N(d\omega)
\end{align}
with partition sum 
\begin{align}
Z_{\beta,m,N}:=\int_{\Omega_{l,N}}e^{-\beta H_{m,N}(\omega)}\,\mu_N(d\omega). 
\end{align}
Clearly, $P_{\beta,m,N}$ and $Z_{\beta,m,N}$ depend also on $\alpha$, $l$, 
and $V$. Usually, we suppress these parameters in the notation. 
Since $V$ is bounded on $[1-\alpha,1+\alpha]$ and $\mu_N(\Omega_{l,N})<\infty$, 
it follows that $Z_{\beta,m,N}<\infty$. 
Lemma \ref{lemma-lower-bound-Z} below shows that 
$Z_{\beta,m,N}>0$ holds as well.

\subsection{Results}

We remark that under the assumptions stated at the beginning of 
Section \ref{section1}, for all $\beta>0$, $m\in\R$, 
$N\in\N$ with $N\ge 5$, $x\in A_2$, and $z\in\n$, one has 
\begin{align}
\label{expec-lz}
E_{P_{\beta,m,N}}[\hat\omega(x+z)-\hat\omega(x)]=lz.
\end{align}
This follows from \eqref{def-OmegaRepr} together with the translational 
invariance of $P_{\beta,m,N}$.
In particular, under $P_{\beta,m,N}$, the distribution of 
$\hat\omega(x+z)-\hat\omega(x)$ is not rotationally invariant. 
Note that $|\hat\omega(x+z)-\hat\omega(x)|$ is bounded uniformly in $N$, 
and thus, equation \eqref{expec-lz} remains true when one takes 
subsequential weak limits as $N\to\infty$. 
As a consequence, any infinite volume Gibbs measure obtained 
as such a subsequential limit is not rotationally invariant. 

We prove a much stronger form of breaking of rotational symmetry. 

\begin{theorem}
\label{main-thm}
Under the assumptions stated at the beginning of Section \ref{section1}, 
there is a constant $m_0=m_0(V)$, such that the following holds:
\begin{align}
\label{claim-main-thm}
\lim_{\beta\to\infty}\sup_{N\ge 5}\sup_{m\ge m_0} \sup_{x\in A_2} \sup_{z\in\n}
E_{P_{\beta,m,N}}
[|\hat\omega(x+z)-\hat\omega(x)-lz|^2 ]=0. 
\end{align}
\end{theorem}

\begin{corollary}
Under the assumptions of Theorem \ref{main-thm}, 
\begin{align}
\lim_{\beta\to\infty}\sup_{N\ge 5}\sup_{m\ge m_0} 
\sup_{x\in A_2} \sup_{z\in\n}
E_{P_{\beta,m,N}}
[|\omega(x+z)-\omega(x)-lz|^21_{\{\omega(x+z)\neq\weg,\omega(x)\neq\weg\}}
]=0. 
\end{align}
\end{corollary}

For every triangle $\Delta\in\T$, $\hat\omega$ is affine linear on 
$\Delta$. Hence, its Jacobian $\nabla\hat\omega$ is constant on $\Delta$; 
we denote by $\nabla\hat\omega(\Delta)$ this constant value. 

\begin{theorem}
\label{variant-main-thm}
Under the assumptions of Theorem \ref{main-thm}, 
\begin{align}
\label{claim-variant-main-thm}
\lim_{\beta\to\infty}
\sup_{N\ge 5}\sup_{m\ge m_0}\sup_{\Delta\in\T_N} 
E_{P_{\beta,m,N}}[|\nabla\hat\omega(\Delta)-l\id|^2]
=0. 
\end{align}
\end{theorem}

Finally, a remark on infinite volume limits. Since the above results  
are uniform in the size $N$ of the underlying lattice, 
the finite-volume results carry over to infinite-volume Gibbs measures obtained 
as subsequential limits as $N\to\infty$.

\paragraph{Organization.} 
In our proof of these results, we proceed as follows. 
In Section \ref{sec-estimate-H}, we compare the Hamiltonian of a configuration $\omega\in\Omega_{l,N}$ with the Hamiltonian of the \emph{standard configuration} $\omega_l$. 
Subsequently, in Section \ref{sec-fin-vol-estimates}, we use these estimates to bound the partition sum from below and the internal energy from above. 
Our proofs rely crucially on the following rigidity estimate. 
We use it both locally (in Lemma \ref{lemma-U0-U1}), and globally (in 
Lemma~\ref{lower-bound-A}).

\begin{theorem}[Friesecke, James, and M{\"u}ller
{\cite[Theorem 3.1]{friesecke-james-mueller2002}}]
\label{thm-friesecke-et-al}
Let $U$ be a bounded Lipschitz domain in $\R^n$, $n\ge 2$. There exists
a constant $C(U)$ with the following property: For each $v\in W^{1,2}(U,\R^n)$
there is an associated rotation $R\in \SO(n)$ such that 
\begin{align}
\label{align-friesecke}
\|\nabla v-R\|_{L^2(U)}\le C(U)\|\dist(\nabla v,\SO(n))\|_{L^2(U)}.
\end{align}
\end{theorem}

We are interested in bounded domains $U\subset\R^2$ which are bounded by 
finitely many pieces of straight lines and in  
continuous functions $v\colon U\to\R^2$ that are piecewise affine
linear with respect to a triangulation of $U$. Note that these
functions belong to $W^{1,2}(U,\R^2)$. 

\begin{remark}
\label{remark-friesecke-const}
The constant $C(U)$ in Theorem \ref{thm-friesecke-et-al} 
is invariant under scaling: $C(\gamma U)=C(U)$ for all $\gamma>0$. 
Indeed, setting $v_\gamma(\gamma x)=\gamma v(x)$ for $x\in U$, we have 
$\nabla v_\gamma(\gamma x)=\nabla v(x)$ and hence 
$\|\nabla v_\gamma-R\|_{L^2(\gamma U)} = \gamma^{n/2} \|\nabla v-R\|_{L^2(U)}$
and $ \|\dist(\nabla v_\gamma,\SO(n))\|_{L^2(\gamma U)} =
\gamma^{n/2}\|\dist(\nabla v,\SO(n))\|_{L^2(U)} $. 
This implies that for the interior $U_N$ of $\Lambda_{1N}$, one can 
choose the constant $C(U_N)$ in Theorem \ref{thm-friesecke-et-al}
as a constant $\celf$ independent of $N$. 
\end{remark}

\section{An estimate for the Hamiltonian}
\label{sec-estimate-H}

We identify $\C$ with $\R^2$. 
In this section, we prove the following. 

\begin{lemma}
\label{lemma-estimate-H}
There exist constants $\cdrei(V)>0$ and $\cneun(V)>0$ such that for all 
$\omega\in\Omega_{l,N}$, one has 
\begin{align}
H_{m,N}(\omega)-H_{m,N}(\omega_l)
\ge \cdrei \sum_{\Delta\in\T_N} 
\dist(l^{-1}\nabla\hat\omega(\Delta),\SO(2))^2
+ (m-\cneun)|\defects(\omega)|.
\end{align}
\end{lemma}

Here and in the rest of the paper, the distance is taken with respect to 
an arbitrary norm $\|{\cdot}\|$ on $2\times 2$-matrices.  

First, we estimate the contribution of the Hamiltonian for single 
triangles. Then, we show that the defects are negligible. 

\subsection{Estimates for individual triangles}

Let $\Delta$ be a triangle in $\R^2$ with corner points $A_1,A_2,A_3$, i.e.\  
the interior of the convex hull of $\{ A_1,A_2,A_3\}$. Let further  
$\omega\colon\R^2\to\R^2$ be the affine linear map that maps 
$0,1,\tau$ to $A_1,A_2,A_3$, respectively. We assume that 
$(A_1,A_2,A_3)$ is positively oriented, i.e.\ $\det\nabla\omega>0$. 
We introduce the sides of the triangle: 
\begin{align}
\label{notation-triangle}
\begin{array}{lll}
\vec{a}_1:=A_3-A_2, & \quad \vec{a}_2:=A_1-A_3, & \quad \vec{a}_3:=A_2-A_1,\\
a_j:=|\vec{a}_j|.   & & 
\end{array}
\end{align}

Recall that $l\Delta_{0,1}$ is an equilateral triangle with side length $l$. 

Throughout, we write $T\asymp S$ for terms $T\ge 0$ and $S\ge 0$ if 
there are uniform constants $c,C>0$ such that 
$cT\le S\le CT$ holds. If the constants depend on the fixed potential $V$,
we write $T\asymp_V S$.

\begin{lemma}
\label{lemma-V-dreieck}
Let $p(l):=2\sqrt{3}V'(l)/l$. For sufficiently small $\alpha>0$ and 
side lengths $a_1,a_2,a_3\in(1-\alpha,1+\alpha)$, one has 
\begin{align}
\sum_{j=1}^3 V(a_j)-3V(l) 
- p(l) (\lambda(\Delta)-\lambda(l\Delta_{0,1}))
\asymp_V \sum_{j=1}^3 (a_j-l)^2 .
\end{align}
\end{lemma}
\begin{proof}
Heron's formula gives the area of the triangle $\Delta$ with side length
$a_1$, $a_2$, and $a_3$ as 
\begin{align}
\lambda(\Delta) & = 
\frac14\sqrt{(a_1+a_2-a_3)(a_2+a_3-a_1)(a_3+a_1-a_2)(a_1+a_2+a_3)}
\nonumber\\
&=:  A(a_1,a_2,a_3). 
\end{align}
The function $A$ is twice continuously differentiable with 
\begin{align}
\frac{\partial A}{\partial a_j}(l,l,l)=\frac{l}{2\sqrt{3}}
\quad\text{ for }j\in\{1,2,3\}. 
\end{align}
All second derivatives of $A(a_1,a_2,a_3)$ are 
bounded for $a_1,a_2,a_3\in(1-\alpha,1+\alpha)$, with $\alpha>0$ small
enough. 
Consequently, 
\begin{align}
\lambda(\Delta)-\lambda(l\Delta_{0,1})
=\frac{l}{2\sqrt{3}} \Big(\sum_{j=1}^3 a_j -3l\Big)
+ \sum_{j=1}^3 O((a_j-l)^2) 
\quad\text{as }a_j\to l. 
\end{align}
Since $V$ is twice differentiable, we get using the last equation
\begin{align}
\sum_{j=1}^3 V(a_j)-3V(l) 
= & V'(l)\Big(\sum_{j=1}^3 a_j -3l\Big)
+\frac12 V''(l) \sum_{j=1}^3 (a_j-l)^2 + \sum_{j=1}^3 o_V((a_j-l)^2)\nonumber\\
 = & \frac{2\sqrt{3}V'(l)}{l}(\lambda(\Delta)-\lambda(l\Delta_{0,1}))
+ V'(l)\sum_{j=1}^3 O((a_j-l)^2) \nonumber\\
& +\frac12 V''(l) \sum_{j=1}^3 (a_j-l)^2 + \sum_{j=1}^3 o_V((a_j-l)^2)
\quad\text{as }a_j\to l. 
\end{align}
By assumption, $\inf_{1-\frac\alpha2\le l\le 1+\frac\alpha2}V''(l)>0$. 
Clearly $V'(1)=0$ implies 
$\sup_{1-\frac\alpha2\le l\le 1+\frac\alpha2}|V'(l)|
\le\frac{\alpha}{2}\sup_{1-\frac\alpha2\le\xi\le 1+\frac\alpha2}|V''(\xi)|=O_V(\alpha)$. 
The claim follows for $\alpha$ small enough. 
\end{proof}

\begin{lemma}
\label{lemma-dreieck-dist}
For sufficiently small $\tilde\alpha>0$ and 
side lengths $a_1,a_2,a_3\in(1-\tilde\alpha,1+\tilde\alpha)$, one has 
\begin{align}
\sum_{j=1}^3 (a_j-1)^2\asymp
\dist(\nabla\omega,\SO(2))^2
\end{align}
with $\omega$ defined before (\ref{notation-triangle}). 
\end{lemma}
\begin{proof}
Let $E_1=0$, $E_2=1$, $E_3=\tau$ denote the corner points of the 
standard equilateral triangle. 
Set $M:=\nabla\omega$; $M$ is constant since $\omega$ is affine linear. 
Consequently, for any cyclic permutation $(i,j,k)$ of $(1,2,3)$, one has 
\begin{align}
\label{id-a-1}
a_i=|\omega(E_j)-\omega(E_k)|=|M(E_j-E_k)|=|Mv_i|,
\end{align} 
where we set $v_i:=E_j-E_k$. Clearly, $|v_i|=1$.  
Now $a_i-1\asymp (a_i-1)(a_i+1)=a_i^2-1$ because $a_i\in(1-\tilde\alpha,
1+\tilde\alpha)$
and $\tilde\alpha$ is small enough. 
Using (\ref{id-a-1}), we obtain
\begin{align}
a_i-1\asymp a_i^2-1 = \langle v_i, M^* Mv_i\rangle - |v_i|^2 
= \langle v_i, (M^* M-\id)v_i\rangle. 
\end{align}
For $Q\in\R^{2\times 2}_\sym$, the set of symmetric $2\times 2$ matrices, set 
$\|Q\|_v:=(\sum_{j=1}^3 \langle v_j,Qv_j\rangle^2)^{1/2}$. Clearly, 
$\|\cdot\|_v$ is a seminorm on $\R^{2\times 2}_\sym$. To see that it is 
a norm, assume that $\|Q\|_v=0$, i.e.\ 
$\langle v_j,Qv_j\rangle=0$ for $j=1,2,3$. Using $v_1+v_2+v_3=0$ and 
the symmetry of $Q$, it follows that 
$\langle v_j,Qv_k\rangle=0$ for all $j,k\in\{1,2,3\}$. Since $v_1,v_2,v_3$ 
span $\R^2$, we conclude $Q=0$. 
Since all norms on $\R^{2\times 2}_\sym$ are equivalent, we have shown 
\begin{align}
\label{sum-a-j-quadrat}
\sum_{j=1}^3 (a_j-1)^2 \asymp \| M^* M-\id\|^2
\end{align}
for any norm $\|{\cdot}\|$. 

We use now the following fact: Assume that $S$ is a compact submanifold of 
$\R^d$, given as a set of zeros 
\begin{align}
S=\{x\in U:f(x)=0\}
\end{align}
for some open set $U\subseteq\R^d$ and some smooth function $f\colon U\to\R^m$, 
$m\le d$. Assume further that $\nabla f$ has rank $m$ on $S$. Then, 
there is a neighborhood $U'\subseteq U$ of $S$ such that for all $x\in U'$, 
\begin{align}
\dist(x,S)\asymp\|f(x)\|. 
\end{align}

We apply this fact to $S=SO(2)$, $U=\{ Q\in\R^{2\times 2}:\det Q>0\}$, 
and $f\colon U\to\R^{2\times 2}_\sym$, $f(Q)=Q^*Q-\id$; its derivative has full 
rank on $S$. For $\tilde\alpha>0$ sufficiently small and 
$|a_j-1|<\tilde\alpha$, 
$j=1,2,3$, $M=\nabla\omega$ is close to $\SO(2)$; recall that 
$\det M>0$ by ($\Omega 4$). Consequently, 
\begin{align}
\|M^*M-\id\|\asymp \dist(M,\SO(2)).
\end{align}
Together with (\ref{sum-a-j-quadrat}), this implies the claim. 
\end{proof}

Combining Lemmas \ref{lemma-V-dreieck} and \ref{lemma-dreieck-dist} 
and scaling with $l$, which is close to $1$, yields the following.
\begin{corollary}
\label{cor-single-triangle}
For sufficiently small $\alpha>0$ and 
side lengths $a_1,a_2,a_3\in(1-\alpha,1+\alpha)$, and 
$1-\alpha/2<l<1+\alpha/2$, one has 
\begin{align}
\sum_{j=1}^3 V(a_j)-3V(l) 
- p(l) (\lambda(\Delta)-\lambda(l\Delta_{0,1}))
\asymp_V \dist(l^{-1}\nabla\omega,\SO(2))^2
\end{align}
with $\omega$ defined before (\ref{notation-triangle}). 
\end{corollary}

\subsection{Contributions from defects}

\begin{definition}
For $x\in A_2$, let $U_0(x):=\{\Delta\in\T:
x\in\mathrm{closure}(\Delta)\}$ denote 
the set of all triangles in $\T$ incident to $x$. Let 
\begin{align}
U_1(x):=\{\Delta\in\T:\text{ all corner points of }\Delta
\text{ are contained in }x+\n+\n\}\setminus U_0(x)
\end{align}
denote the ``second layer'' of triangles around $x$. In the special case
$x=0$, we abbreviate $U_0:=U_0(0)$ and $U_1:=U_1(0)$. 
\end{definition}

\begin{lemma}
\label{lemma-U0-U1}
There exists a constant $\cvier>0$ such that for all $\omega\in\Omega_{l,N}$ 
with $\omega(0)=\weg$, one has 
\begin{align}
\sum_{\Delta\in U_0}\dist(\nabla\hat\omega(\Delta),\SO(2))^2 
\le\cvier\sum_{\Delta\in U_1}\dist(\nabla\hat\omega(\Delta),\SO(2))^2 .
\end{align}
\end{lemma}
\begin{proof}
We apply the theorem by Friesecke et al.\ (Theorem 
\ref{thm-friesecke-et-al}) to the interior $U$ of 
$\bigcup_{\Delta\in U_1}\mathrm{closure}(\Delta)$, using 
\begin{align}
\lambda(\Delta_{0,1})\sum_{\Delta\in U_1}\dist(\nabla\hat\omega(\Delta),\SO(2))^2 
=\|\dist(\nabla\hat\omega,\SO(2))\|_{L^2(U)}^2. 
\end{align}
Hence there exists a rotation $R\in \SO(2)$ with 
\begin{align}
\sum_{\Delta\in U_1}\|\nabla\hat\omega(\Delta)-R\|^2 
\le C(U)\sum_{\Delta\in U_1}\dist(\nabla\hat\omega(\Delta),\SO(2))^2.
\end{align} 
We introduce the piecewise affine linear map 
$\sigma\colon\conv(\n+\n)\to\R^2$, 
$\sigma(x)=\hat\omega(x)-\hat\omega(0)-Rx$. 
The map $\sigma$ belongs to the finite-dimensional vector space $W$ 
of all continuous piecewise 
affine linear maps $\sigma'\colon\conv(\n+\n)\to\R^2$ which are affine linear 
on the closure of every $\Delta$ 
and satisfy $\sigma'(0)=0=|\n|^{-1}\sum_{\tau\in\n}\sigma'(\tau)$. 
By definition of $\sigma$, one has 
\begin{align}
\sum_{\Delta\in U_1}\|\nabla\hat\omega(\Delta)-R\|^2 
=\sum_{\Delta\in U_1}\|\nabla\sigma(\Delta)\|^2. 
\end{align}
If $Q(\sigma'):=\sum_{\Delta\in U_1}\|\nabla\sigma'(\Delta)\|^2=0$ for some 
$\sigma'\in W$, then $\sigma'=0$. Indeed, we obtain first that $\sigma'$ is 
constant on all triangles in $U_1$. The value $\sigma'(0)=0$ is the 
average of this constant; hence the constant vanishes.
Consequently, the quadratic form 
$Q\colon W\to\R$ is positive definite. Since $W$ is finite-dimensional, 
any quadratic form on $W$ is bounded from above by a constant multiple 
of $Q$. In particular, for some constant $\cfuenf>0$ and any $\sigma'\in W$, 
\begin{align}
\sum_{\Delta\in U_0}\|\nabla\sigma'(\Delta)\|^2
\le \cfuenf \sum_{\Delta\in U_1}\|\nabla\sigma'(\Delta)\|^2. 
\end{align}
For the special case $\sigma'=\sigma$ this yields
\begin{align}
& \sum_{\Delta\in U_0}\dist(\nabla\hat\omega(\Delta),\SO(2))^2 
\le \sum_{\Delta\in U_0} \|\nabla\hat\omega(\Delta)-R\|^2
= \sum_{\Delta\in U_0}\|\nabla\sigma(\Delta)\|^2
\nonumber\\
& \le \cfuenf \sum_{\Delta\in U_1}\|\nabla\sigma(\Delta)\|^2
\le \cfuenf C(U) \sum_{\Delta\in U_1}\dist(\nabla\hat\omega(\Delta),\SO(2))^2 .
\end{align}
\end{proof}

We call the triangle $\Delta_{x,z}\in\T_N$ {\it present} in the configuration 
$\omega\in\Omega_{l,N}$ if 
$\omega(x)\neq\weg$, $\omega(x+z)\neq\weg$, and $\omega(x+\tau z)\neq\weg$. 
Let 
\begin{align}
\T_N^\pres(\omega):=\{\Delta\in\T_N:\Delta\text{ is present in }\omega\}. 
\end{align}

If there is a defect at $x$, then by assumption ($\Omega$2), all triangles
in the second layer $U_1(x)$ are present. 

\begin{lemma}
\label{lemma-sum-defects}
For all $\omega\in\Omega_{l,N}$, one has 
\begin{align}
\sum_{\Delta\in\T_N}\dist(\nabla\hat\omega(\Delta),\SO(2))^2 
\asymp \sum_{\Delta\in\T_N^\pres(\omega)}\dist(\nabla\hat\omega(\Delta),\SO(2))^2 , 
\end{align}
where the constants for $\asymp$ can be chosen independently of $\omega$. 
\end{lemma}
\begin{proof}
The bound ``$\ge$'' holds trivially. For the upper bound, we proceed by 
splitting the sum as follows
\begin{align}
\sum_{\Delta\in\T_N}\dist(\nabla\hat\omega(\Delta),\SO(2))^2 
= & \sum_{\Delta\in\T_N^\pres(\omega)}\dist(\nabla\hat\omega(\Delta),\SO(2))^2 
\nonumber\\
& + \sum_{x\in\defects(\omega)}\sum_{\Delta\in U_0(x)}
\dist(\nabla\hat\omega(\Delta),\SO(2))^2 .
\end{align}
By Lemma \ref{lemma-U0-U1}, 
\begin{align}
 \sum_{x\in\defects(\omega)}\sum_{\Delta\in U_0(x)}
\dist(\nabla\hat\omega(\Delta),\SO(2))^2 
\le\cvier \sum_{x\in\defects(\omega)}
\sum_{\Delta\in U_1(x)}\dist(\nabla\hat\omega(\Delta),\SO(2))^2
\nonumber\\ 
= \cvier \sum_{\Delta\in\T_N^\pres(\omega)}  \sum_{x\in\defects(\omega)} 
1_{U_1(x)}(\Delta)\; \dist(\nabla\hat\omega(\Delta),\SO(2))^2 .
\end{align}
Now, $\sum_{x\in\defects(\omega)} 1_{U_1(x)}(\Delta)\le 9$ 
for all $\omega\in\Omega_{l,N}$ and $\Delta\in\T_N$. The claim follows. 
\end{proof}

\subsection{Proof of Lemma \ref{lemma-estimate-H}}

Let $\omega\in\Omega_{l,N}$. 
For $x\in I_N$ and $y\in A_2$ with $x\sim y$, $\omega(x)\neq\weg$ and 
$\omega(y)\neq\weg$, we call the undirected edge $\{x,y\}$ 
\begin{itemize}
\item a {\it boundary edge} with respect to $\omega$ if there exists 
$z\in A_2$ with $z\sim x$, $z\sim y$, and $\omega(z)=\weg$;
\item an {\it inner edge} with respect to $\omega$ otherwise. 
\end{itemize}
We denote the set of boundary and inner edges with respect to 
$\omega$ by $\partial E_N(\omega)$ and $E_N^\circ(\omega)$, respectively.

\smallskip\noindent\begin{proof}[Proof of Lemma \ref{lemma-estimate-H}]
For all $x\in I_N$ and $y\in I_{N+1}$ with $x\sim y$, one has 
$|\omega_l(x)-\omega_l(y)|=l$ for the standard configuration $\omega_l$. 
Thus, any edge $\{x,y\}$ contributes the amount $V(l)$ to $H_{m,N}(\omega_l)$. 

Let $\omega\in\Omega_{l,N}$. For $\Delta\in \T_N^\pres(\omega)$, let
$a_j(\Delta)$, $j=1,2,3$, denote the side lengths of the triangle 
$\omega(\Delta)$. 
For any $x\in I_N$ with $\omega(x)=\weg$, there are 6 edges incident 
to $x$ which are neither boundary edges nor inner edges with respect
to $\omega$. Consequently, we obtain
\begin{align}
& H_{m,N}(\omega)-H_{m,N}(\omega_l) + (6V(l)-m) |\defects(\omega)|
\nonumber\\
= & \sum_{\{x,y\}\in\partial E_N(\omega)\cup E_N^\circ(\omega)}
[V(|\omega(x)-\omega(y)|)-V(l)] \nonumber\\
= & \frac12\sum_{\Delta\in \T_N^\pres(\omega)}
\Big( \sum_{j=1}^3V(a_j(\Delta))-3V(l) \Big)
+\frac12\sum_{\{x,y\}\in\partial E_N(\omega)}[V(|\omega(x)-\omega(y)|)-V(l)]
. 
\label{align-est-H1}
\end{align}
For the last equation, note that the first term counts only half of the 
contribution from boundary edges, although their contribution needs to 
be fully counted. 

Since $|V|$ is bounded on $(1-\alpha,1+\alpha)$ by some constant 
$\csechs(V)$, we get the following estimate for the last sum in 
(\ref{align-est-H1}):  
\begin{align}
\sum_{\{x,y\}\in\partial E_N(\omega)}[V(|\omega(x)-\omega(y)|)-V(l)]
\ge &-2\csechs(V)|\partial E_N(\omega)| \cr
= & -12\csechs(V)|\defects(\omega)|
. 
\label{align-est-H5}
\end{align}

We now estimate the first sum on the right hand side of (\ref{align-est-H1}). 
By $(\Omega 1)$, one has $a_j(\Delta)\in(1-\alpha,1+\alpha)$ for all 
$\Delta\in \T_N^\pres(\omega)$. Thus, Corollary \ref{cor-single-triangle} 
and Lemma \ref{lemma-sum-defects} yield
\begin{align}
& \sum_{\Delta\in \T_N^\pres(\omega)}
\Big( \sum_{j=1}^3V(a_j(\Delta))-3V(l) 
- p(l) (\lambda(\hat\omega(\Delta))-\lambda(l\Delta_{0,1})) \Big) \nonumber\\
\asymp_V & \sum_{\Delta\in \T_N^\pres(\omega)}
\dist(l^{-1}\nabla\hat\omega(\Delta),\SO(2))^2 \nonumber\\
\asymp_V & \sum_{\Delta\in \T_N}
\dist(l^{-1}\nabla\hat\omega(\Delta),\SO(2))^2
.
\label{align-est-H2}
\end{align}
Note that by $(\Omega 3)$ and periodicity \eqref{def-OmegaRepr}, 
$\hat\omega$ maps any measurable set of representatives of $\C$ modulo
$NA_2$ onto a set having the Lebesgue measure $\lambda(\Lambda_{lN})$. 
Consequently, 
\begin{align}
\sum_{\Delta\in \T_N}
 (\lambda(\hat\omega(\Delta))-\lambda(l\Delta_{0,1})) 
= \lambda(\Lambda_{lN})-\lambda(l\Lambda_{1N}) =0.
\end{align}
Hence, since the area of any defective hexagon is uniformly bounded, 
we find 
\begin{align}
\left|\sum_{\Delta\in \T_N^\pres(\omega)}
 (\lambda(\hat\omega(\Delta))-\lambda(l\Delta_{0,1})) \right|
= & \left|\sum_{\Delta\in\T_N\setminus\T_N^\pres(\omega)}
 (\lambda(\hat\omega(\Delta))-\lambda(l\Delta_{0,1})) \right|
\nonumber\\
\le & \, \csieben |\defects(\omega)|
\label{align-est-H3}
\end{align}
with a uniform constant $\csieben>0$.
Combining this with (\ref{align-est-H2}), we obtain
\begin{align}
& \sum_{\Delta\in \T_N^\pres(\omega)}
\Big( \sum_{j=1}^3V(a_j(\Delta))-3V(l) \Big)
\nonumber\\
\ge & \cacht \sum_{\Delta\in \T_N}
\dist(l^{-1}\nabla\hat\omega(\Delta),\SO(2))^2
- \csieben|p(l)|\cdot|\defects(\omega)|
\label{align-est-H4}
\end{align}
with a constant $\cacht>0$. 

Note that $p(l)=2\sqrt{3}V'(l)/l$ as defined 
in Lemma \ref{lemma-V-dreieck} is bounded for 
$l\in(1-\alpha/2,1+\alpha/2)$. 
Combining (\ref{align-est-H1}), (\ref{align-est-H5}), and
(\ref{align-est-H4}) yields the claim. 
\end{proof}

\section{Uniform finite-volume estimates}
\label{sec-fin-vol-estimates}

\subsection{Lower bound for the partition sum}

\begin{lemma}
\label{lemma-lower-bound-Z}
For all $\varepsilon>0$, there exists $r=r(\varepsilon)>0$ such that for all
$\beta>0$, $m,N$, one has 
\begin{align}
\label{assertion-lemma-lower-bound-Z}
Z_{\beta,m,N}\ge \frac{\lambda(\Lambda_{lN})}{\pi r^2}
e^{-|I_N|(3\beta\varepsilon-\log(\pi r^2))} e^{-\beta H_{m,N}(\omega_l)}.   
\end{align}
\end{lemma}
\begin{proof}
For $r>0$, we consider the set of configurations 
which are, up to translations, sufficiently close to the standard 
configuration and have no defects
\begin{align}
S_{r,l,N}:=
\{\omega\in\underline\Omega_{l,N}:\omega(x)\neq\weg \text{ and }
|\omega(x)-\omega(0)-\omega_l(x)|<r \text{ for all } x\in A_2\}.
\end{align}
Let $\varepsilon>0$. Since $V$ is continuous, for all sufficiently small
$r>0$, for all $N$, for all 
$\omega\in S_{r,l,N}$ and all $x,y\in A_2$ with $x\sim y$, one has
$|V(|\omega(x)-\omega(y)|)-V(l)|<\varepsilon$. 
Consequently, 
$|H_{m,N}(\omega)-H_{m,N}(\omega_l)|\le 3|I_N|\varepsilon$
for all $\omega\in S_{r,l,N}$ and we conclude for all $\beta>0$
\begin{align}
Z_{\beta,m,N}\ge & \int_{S_{r,l,N}}e^{-\beta H_{m,N}(\omega)}\,\mu_N(d\omega)
\ge e^{-3\beta |I_N|\varepsilon} e^{-\beta H_{m,N}(\omega_l)} \mu_N(S_{r,l,N}).
\end{align}
We now argue that $S_{r,l,N}\subseteq\Omega_{l,N}$ for $r\in(0,\alpha/4)$. 
Using $|l-1|<\alpha/2$, we get for all $\omega\in S_{r,l,N}$ 
and $x,y\in A_2$ with $x\sim y$,
\begin{align}
\big||\omega(x)-\omega(y)|-1\big|
\le & \big||\omega(x)-\omega(y)|-l\big| + |l-1|
\nonumber\\
< & \big||\omega(x)-\omega(y)|-|\omega_l(x)-\omega_l(y)|\big| +\alpha/2
\nonumber\\
< & 2r +\alpha/2 \le \alpha.
\end{align}
Hence, condition $(\Omega 1)$ is satisfied. Condition $(\Omega 2)$ is satisfied
by absence of defects in $S_{r,l,N}$. 

To see that $\hat\omega$ is one-to-one, note that for sufficiently small $r$ 
and $\omega\in S_{r,l,N}$, the Jacobi matrix 
$\nabla\hat\omega$ is close to the identity matrix and hence 
$\langle v,\nabla\hat\omega(x)v\rangle>0$ for all $v\in\R^2\setminus\{ 0\}$
and all $x\in\R^2$ for which $\hat\omega$ is differentiable at $x$. 
Further, the map $\hat\omega$ is onto. This is a consequence of the following 
topological fact. Consider a lattice $\Gamma\subset\R^2$ of rank 2. 
Then, every continuous map $f\colon\R^2\to\R^2$ with $f(x+y)=f(x)+y$ for all 
$x\in\R^2$ and $y\in\Gamma$ is onto. This shows that condition 
$(\Omega 3)$ is fulfilled. 

Condition $(\Omega 4)$ is satisfied 
for $\omega_l$ and translations of it, and consequently also for 
$\omega\in S_{r,l,N}$ for $r$ sufficiently small. We conclude 
$S_{r,l,N}\subseteq\Omega_{l,N}$. Thus, 
$\mu_N(S_{r,l,N})=(\pi r^2)^{|I_N|-1}\lambda(\Lambda_{lN})$ by 
the definition of $\mu_N$, since integration 
over $\omega(x)$ for all $x\neq 0$ given $\omega(0)$ yields the 
factor $\pi r^2$ and integration 
over $\omega(0)$ yields the volume $\lambda(\Lambda_{lN})$. 
Consequently, we get the assertion (\ref{assertion-lemma-lower-bound-Z}) 
of the lemma. 
\end{proof}

\subsection{Upper bound for the internal energy}

For $\omega\in\Omega_{l,N}$, we abbreviate 
\begin{align}
A_{m,l,N}(\omega):=H_{m,N}(\omega)-H_{m,N}(\omega_l). 
\end{align}
Recall that $U_N=(0,N)+\tau^2(0,N)$.

\begin{lemma}
\label{lower-bound-A}
There exists a constant $\czehn>0$ such that for all $\beta>0$,
$m\in\R$,  
$N\ge 5$, and $\omega\in\Omega_{l,N}$, one has 
\begin{align}
A_{m,l,N}(\omega) -(m-\cneun)|\defects(\omega)|
\ge \czehn \| l^{-1}\nabla\hat\omega-\id\|^2_{L^2(U_N)}.
\end{align}
\end{lemma}
\begin{proof}
Recall that all triangles in $\T_N$ have the same Lebesgue measure. 
Using this and Lemma \ref{lemma-estimate-H}, we get 
\begin{align}
A_{m,l,N}(\omega)-(m-\cneun)|\defects(\omega)|
\ge & \cdrei \sum_{\Delta\in\T_N} 
\dist(l^{-1}\nabla\hat\omega(\Delta),\SO(2))^2 \nonumber\\
= & \cdrei  \lambda(\Delta_{0,1})^{-1} \sum_{\Delta\in\T_N}
\lambda(\Delta) \dist(l^{-1}\nabla\hat\omega(\Delta),\SO(2))^2
\nonumber\\
= &\cdrei  \lambda(\Delta_{0,1})^{-1} 
\| \dist(l^{-1}\nabla \hat\omega,\SO(2))\|^2_{L^2(U_N)}. 
\label{estA1}
\end{align}
By Theorem \ref{thm-friesecke-et-al} and Remark 
\ref{remark-friesecke-const} there exists a random rotation 
$R_N(\omega)\in\SO(2)$ such that one has 
\begin{align}
\| \dist(l^{-1}\nabla\hat\omega,\SO(2))\|^2_{L^2(U_N)}
\ge\celf^{-1} \| l^{-1}\nabla\hat\omega-R_N(\omega)\|^2_{L^2(U_N)}.
\label{estA2}
\end{align}
Combining (\ref{estA1}) and (\ref{estA2}) yields 
\begin{align}
& A_{m,l,N}(\omega)-(m-\cneun)|\defects(\omega)|
\ge  \czehn  \| l^{-1}\nabla\hat\omega-R_N(\omega)\|^2_{L^2(U_N)}\nonumber\\
& =  \czehn \left( \| l^{-1}\nabla\hat\omega-\id\|^2_{L^2(U_N)} 
+2\langle l^{-1}\nabla\hat\omega-\id, \id-R_N(\omega)\rangle
+\|\id-R_N(\omega)\|^2_{L^2(U_N)} \right) \cr
& \ge  \czehn \left( \| l^{-1}\nabla\hat\omega-\id\|^2_{L^2(U_N)} 
+2\langle l^{-1}\nabla\hat\omega-\id, \id-R_N(\omega)\rangle
 \right)
\label{estA3}
\end{align}
with a constant $\czehn>0$. We introduce the periodic function
$\sigma_\omega(x):=l^{-1}\hat\omega(x)-x$ for $x\in\C$. Its derivative equals
$\nabla\sigma_\omega=l^{-1}\nabla\hat\omega-\id$. 
By the fundamental theorem of calculus, derivatives of periodic 
functions are orthogonal in $L^2$ to any constant function. Thus, 
the scalar product on the right-hand side in (\ref{estA3}) vanishes, and 
we get the claim.
\end{proof}

\begin{lemma}
\label{lemma-est-expect-part2}
There exists a uniform constant $\csiebzehn$ such that the following holds:
For all $\delta>0$, there exist $\czwoelf>0$ and $\cdreizehn\in\R$ 
such that for any $\beta\ge\csiebzehn$, $m\ge m_0:=\cneun+1$ (with $\cneun$ as in 
Lemma \ref{lemma-estimate-H}) and any $N\ge5$,
one has 
\begin{align}
\label{upper-bound-expec-A-delta}
\frac{1}{|\T_N|} 
E_{P_{\beta,m,N}}[ A_{m,l,N}]
\le \frac{\delta}{2} + \czwoelf 
\exp\left\{ |I_N| \left(-\frac{\beta\delta}{8} -\log\beta + \cdreizehn
\right)\right\}. 
\end{align}
As a consequence, 
\begin{align}
\label{upper-bound-limsup}
\limsup_{\beta\to\infty}\sup_{N\ge 5}\sup_{m\ge m_0}
\frac{1}{|\T_N|}E_{P_{\beta,m,N}}[ A_{m,l,N}(\omega)]\le 0. 
\end{align}
\end{lemma}
\begin{proof}
Let $\delta>0$. We calculate  
\begin{align}
& Z_{\beta,m,N} E_{P_{\beta,m,N}}[ A_{m,l,N}(\omega)]
=
\int_{\Omega_{l,N}}A_{m,l,N}(\omega) e^{-\beta H_{m,N}(\omega)} 
\, \mu_N(d\omega)\nonumber\\
= & e^{-\beta H_{m,N}(\omega_l)} 
\int_{\Omega_{l,N}}A_{m,l,N}(\omega) 
e^{-\beta A_{m,l,N}(\omega)} 
\, \mu_N(d\omega).
\label{upper-bound-A-11}
\end{align}
Next, we split the domain of integration into 
\begin{align}
\Omega_{l,N}^{>\delta}:=\{\omega\in\Omega_{l,N}:\;
A_{m,l,N}(\omega)>\delta |I_N|\} \quad\text{ and }\quad 
\Omega_{l,N}^{\le\delta}:=\Omega_{l,N}\setminus\Omega_{l,N}^{>\delta}.
\end{align}
For the latter domain, we estimate
\begin{align}
\int_{\Omega_{l,N}^{\le\delta}}A_{m,l,N}(\omega) 
e^{-\beta A_{m,l,N}(\omega)} 
\, \mu_N(d\omega)
\le & \delta |I_N| Z_{\beta,m,N} e^{\beta H_{m,N}(\omega_l)} \cr
= & \frac\delta2 |\T_N| Z_{\beta,m,N} e^{\beta H_{m,N}(\omega_l)} . 
\label{upper-bound-A-12}
\end{align}
For the remaining part, we first apply the inequality $xe^{-x}\le e^{-x/2}$ 
with $x=\beta A_{m,l,N}$, then we use the exponential Chebyshev inequality. 
This yields 
\begin{align}
& \int_{\Omega_{l,N}^{>\delta}}A_{m,l,N}(\omega) 
e^{-\beta A_{m,l,N}(\omega)} 
\, \mu_N(d\omega) 
\le \frac1\beta\int_{\Omega_{l,N}^{>\delta}} e^{-\beta A_{m,l,N}(\omega)/2} 
\, \mu_N(d\omega) \nonumber\\
\le & \frac1\beta\int_{\Omega_{l,N}} e^{\beta(A_{m,l,N}(\omega)-\delta |I_N|)/4}
e^{-\beta A_{m,l,N}(\omega)/2} 
\, \mu_N(d\omega)\nonumber\\
= & \frac{e^{-\beta\delta |I_N|/4}}{\beta}
\int_{\Omega_{l,N}} e^{-\beta A_{m,l,N}(\omega)/4}
\, \mu_N(d\omega).
\label{upper-bound-A-13}
\end{align}
Lemma \ref{lower-bound-A} implies
\begin{align}
& \int_{\Omega_{l,N}} e^{-\beta A_{m,l,N}(\omega)/4}
\, \mu_N(d\omega) \nonumber\\
\le &  
\int_{\Omega_{l,N}} 
\exp\left\{-\beta \frac{\czehn}{4} \| l^{-1}\nabla\hat\omega-\id\|^2_{L^2(U_N)}
-\frac{\beta}{4}(m-\cneun)|\defects(\omega)|\right\}
\, \mu_N(d\omega).
\label{upper-bound-A-1}
\end{align}
We use again the notation $\sigma_\omega(x):=l^{-1}\hat\omega(x)-x$ 
for $x\in\C$:
\begin{align}
\| l^{-1}\nabla\hat\omega-\id\|^2_{L^2(U_N)}
= & \| \nabla\sigma_\omega \|^2_{L^2(U_N)} \nonumber\\
= & \sum_{\Delta\in\T_N}\| \nabla\sigma_\omega \|^2_{L^2(\Delta)}
=\lambda(\Delta_{0,1})\sum_{\Delta\in\T_N} 
\|\nabla\sigma_\omega(\Delta)\|^2 . 
\label{ident-sigma}
\end{align}
Take an equilateral triangle $\Delta\in\T_N$ 
with corner points $A$, $B$, and $C$. 
We claim that 
\begin{align}
\|\nabla\sigma_\omega(\Delta)\|^2 \ge \cvierzehn 
\left( \|\sigma_\omega(A)-\sigma_\omega(B)\|^2 +
\|\sigma_\omega(B)-\sigma_\omega(C)\|^2 +
\|\sigma_\omega(C)-\sigma_\omega(A)\|^2 \right)
\label{lower-bound-sigma}
\end{align}
with a constant $\cvierzehn>0$ not depending on the choice of $\Delta$.
Since $\sigma_\omega$ is affine linear on $\Delta$, the claim reduces to
showing for any matrix $M\in\R^{2\times 2}$
\begin{align}
\|M\|^2 \ge \cvierzehn 
\left( \|MA-MB\|^2 + \|MB-MC\|^2 + \|MC-MA\|^2 \right).
\label{lower-bound-M}
\end{align}
Note that translating $\Delta$ does not change the claim. 
Thus, we can reduce the claim further to the special cases $\Delta=\Delta_{0,1}$ 
and $\Delta=\tau\Delta_{0,1}$. 
Since both sides in (\ref{lower-bound-M}) are a square of a 
matrix norm on $2\times 2$-matrices, and all such norms  
are equivalent, the claim (\ref{lower-bound-sigma}) follows. 

We bound \eqref{upper-bound-A-1} further from above using \eqref{ident-sigma} and \eqref{lower-bound-sigma} to obtain the upper bound 
\begin{align}
\int_{\Omega_{l,N}} 
\exp\Big\{-\beta \cfuenfzehn \sum_{\substack{x\in I_N,\, y\in A_2\\x\sim y}}
\|\sigma_\omega(x)-\sigma_\omega(y)\|^2
-\frac{\beta}{4}(m-\cneun)|\defects(\omega)|\Big\}
\, \mu_N(d\omega)
\label{upper-bound-A-2}
\end{align}
with a uniform constant $\cfuenfzehn>0$. 
By partitioning $\Omega_{l,N}$ according to the set $D\subset I_N$ of defects, \eqref{upper-bound-A-2} is equal to
\begin{align}
\sum_{D\subset I_N}
e^{-{\beta}\left(m-\cneun\right)\left|D\right|/4}
\int_{\{\defects(\omega)=D\}} 
\exp\Big\{-\beta \cfuenfzehn \sum_{\substack{x\in I_N,\, y\in A_2\\x\sim y}}
\|\sigma_\omega(x)-\sigma_\omega(y)\|^2\Big\}
\, \mu_N(d\omega).
\label{upper-bound-A-3}
\end{align}
By $(\Omega2)$, defects are isolated in $I_N$. 
Whence, for each set $D$ of defects, we can choose a spanning tree $\mathcal S$ of $I_N\setminus \defects(\omega)$. 
We bound \eqref{upper-bound-A-3} from above by restricting the sum of pairs $x\sim y$ to edges $\{x,y\}$ of $\mathcal S$,
\begin{align}
& \int_{\Omega_{l,N}} e^{-\beta A_{m,l,N}(\omega)/4}
\, \mu_N(d\omega) \nonumber\\
\le &  
\sum_{D\subset I_N}
e^{-{\beta}\left(m-\cneun\right)\left|D\right|/4}
\int_{\{\defects(\omega)=D\}} 
\exp\Big\{-\beta \cfuenfzehn \sum_{\{x,y\}\in \mathcal S}
\|\sigma_\omega(x)-\sigma_\omega(y)\|^2\Big\}
\, \mu_N(d\omega)\nonumber\\
=&\sum_{D\subset I_N} e^{-{\beta}\left(m-\cneun\right)\left|D\right|/4}
\left(\int_{\R^2}e^{-\beta \cfuenfzehn l^{-2}\|u\|^2}\lambda(du)\right)^{|I_N|-|D|-1}
\lambda(\Lambda_{lN}),
\label{upper-bound-A-4}
\end{align}
where the factor $\lambda(\Lambda_{lN})$ stems from integrating the root 
of $\mathcal S$ over the set of representatives $\Lambda_{lN}$ 
of $\C/lNA_2$;
a Gaussian integral arises for each of the $|I_N|-|D|-1$ edges of 
$\mathcal S$. 
There exists a uniform constant $\csechzehn>0$ such that 
\begin{align}
\int_{\R^2}e^{-\beta \cfuenfzehn l^{-2}\|u\|^2}\lambda(du)
\le \frac\csechzehn{2\beta}, 
\end{align}
and hence 
\begin{align}
& \int_{\Omega_{l,N}} e^{-\beta A_{m,l,N}(\omega)/4}
\, \mu_N(d\omega) \nonumber\\
\le &  
\left(\frac\csechzehn{2\beta}\right)^{|I_N|-1}
\lambda(\Lambda_{lN}) 
\sum_{D\subset I_N}
\exp\left\{-\left(\frac{\beta}{4}\left(m-\cneun\right)+\log\left(\frac\csechzehn{2\beta}\right)\right)\left|D\right|\right\}
.
\end{align}
Take a uniform constant $\csiebzehn$ so large that for all 
$\beta\ge\csiebzehn$ and $m\ge m_0=\cneun+1$
one has 
\begin{align}
\frac{\beta}{4}\left(m-\cneun\right)+\log\left(\frac\csechzehn{2\beta}\right)
\ge 0. 
\end{align}
For these $\beta$ and $m$, we get 
\begin{align}
&\sum_{D\subset I_N}
\exp\left\{-\left(\frac{\beta}{4}\left(m-\cneun\right)+\log\left(\frac\csechzehn{2\beta}\right)\right)\left|D\right|\right\} \nonumber \\
\le&\left(1+\exp\left\{-\frac{\beta}{4}\left(m-\cneun\right)-\log\left(\frac\csechzehn{2\beta}\right)\right\}\right)^{|I_N|}
\le 2^{|I_N|}.
\end{align}
Thus, 
\begin{align}
\int_{\Omega_{l,N}} e^{-\beta A_{m,l,N}(\omega)/4}
\, \mu_N(d\omega) 
\le 
2\left(\frac\csechzehn{\beta}\right)^{|I_N|-1}
\lambda(\Lambda_{lN}).
\label{upper-bound-A-14}
\end{align}
We combine \eqref{upper-bound-A-11} with \eqref{upper-bound-A-12}, \eqref{upper-bound-A-13}, and \eqref{upper-bound-A-14} to obtain 
\begin{align}
E_{P_{\beta,m,N}}[ A_{m,l,N}(\omega)]
\le 
\frac\delta2 |\T_N|+
\frac{2}{\csechzehn}\frac{e^{-\beta\delta |I_N|/4}}{ Z_{\beta,m,N}}
\left(\frac{\csechzehn}{\beta}\right)^{|I_N|}
\lambda(\Lambda_{lN}) e^{-\beta H_{m,N}(\omega_l)}
.
\end{align}
Next, we insert the lower bound for the partition sum from 
Lemma \ref{lemma-lower-bound-Z} with $\varepsilon:=\delta/24$ and 
$r=r(\varepsilon)$. Using also $|\T_N|\ge 1$, we obtain
\begin{align}
E_{P_{\beta,m,N}}[ A_{m,l,N}(\omega)]
\le & \frac\delta2 |\T_N|+ 
\frac{2}{\csechzehn}\frac{e^{-\beta\delta |I_N|/4}}{ 
\pi^{-1} r^{-2}
e^{-|I_N|(3\beta\varepsilon-\log(\pi r^2))}}
\left(\frac{\csechzehn}{\beta}\right)^{|I_N|}
\cr
\le & \frac\delta2 |\T_N|+ \czwoelf |\T_N|
\exp\left\{ |I_N| \left( 
- \frac{\beta\delta}{8} - \log\beta + \cdreizehn  \right)\right\}
\end{align}
with constants $\czwoelf>0$ and $\cdreizehn\in\R$ depending on $\delta$. 
This yields Claim \eqref{upper-bound-expec-A-delta}. 

For any given $\delta>0$, 
$-\beta\delta/8 -\log\beta + \cdreizehn(\delta)\to -\infty$ as 
$\beta\to\infty$. Consequently, Claim \eqref{upper-bound-limsup} follows. 
\end{proof}

\subsection{Proof of the main results}

\noindent
\begin{proof}[Proof of Theorem \ref{variant-main-thm}]
The claim follows if we show 
\begin{align}
\label{main-reduced2}
\lim_{\beta\to\infty}
\sup_{N\ge 5}\sup_{m\ge m_0}\frac{1}{|\T_N|}\sum_{\Delta\in\T_N} 
E_{P_{\beta,m,N}}[|\nabla\hat\omega(\Delta)-l\id|^2]
=0. 
\end{align}
This can be seen as follows: 
For $x\in A_2$, let $\theta_x\colon\Omega_{l,N}\to\Omega_{l,N}$, 
$\theta_x\omega(y)=\omega(y-x)$ for $y\in A_2$, denote the shift operator. 
For any $x\in A_2$, 
$P_{\beta,m,N}$ is invariant under $\theta_x$. Consequently, for any 
$\tilde\Delta\in\T_N$ and $x\in I_N$, we get 
\begin{align}
\label{delta-shift}
E_{P_{\beta,m,N}}[|\nabla\hat\omega(\tilde\Delta+x)-l\id|^2]
= & 
E_{P_{\beta,m,N}}[|\nabla\hat\omega(\tilde\Delta)-l\id|^2]. 
\end{align}
For any $\Delta_1\in\T_N$, the set 
$\big\{ \Delta=\tilde\Delta+x : 
\tilde\Delta\in\{\Delta_1,\tau\Delta_1\}, x\in I_N \big\}$
modulo translations by elements of $NA_2$ runs over all 
elements of $\T_N$. 
Using this first and then using (\ref{delta-shift}) yields 
\begin{align}
\sum_{\Delta\in\T_N}E_{P_{\beta,m,N}}[|\nabla\hat\omega(\Delta)-l\id|^2]
= & \sum_{\tilde\Delta\in\{\Delta_1,\tau\Delta_1\}}
\sum_{x\in I_N} E_{P_{\beta,m,N}}[|\nabla\hat\omega(\tilde\Delta+x)-l\id|^2]
\nonumber\\
= & \sum_{\tilde\Delta\in\{\Delta_1,\tau\Delta_1\}}
\sum_{x\in I_N} 
E_{P_{\beta,m,N}}[|\nabla\hat\omega(\tilde\Delta)-l\id|^2]
\nonumber\\
\ge & \sum_{x\in I_N} 
E_{P_{\beta,m,N}}[|\nabla\hat\omega(\Delta_1)-l\id|^2]
\cr 
= & 
|I_N| E_{P_{\beta,m,N}}[|\nabla\hat\omega(\Delta_1)-l\id|^2]
. 
\end{align}
Since $2|I_N|=|\T_N|$, (\ref{main-reduced2}) implies 
Claim (\ref{claim-variant-main-thm}). 

To prove (\ref{main-reduced2}), we consider  
\begin{align}
l^{-2}\sum_{\Delta\in\T_N}\lambda(\Delta)
E_{P_{\beta,m,N}}[|\nabla\hat\omega(\Delta)-l\id|^2]
= & \sum_{\Delta\in\T_N}\lambda(\Delta)
E_{P_{\beta,m,N}}[|l^{-1}\nabla\hat\omega(\Delta)-\id|^2]\nonumber\\
= & E_{P_{\beta,m,N}}[\| l^{-1}\nabla\hat\omega-\id\|^2_{L^2(U_N)}]. 
\end{align}
Lemma \ref{lower-bound-A} implies 
\begin{align}
0\le  l^{-2}\sum_{\Delta\in\T_N}\lambda(\Delta)
E_{P_{\beta,m,N}}[|\nabla\hat\omega(\Delta)-l\id|^2]
\le \czehn^{-1} E_{P_{\beta,m,N}}[A_{m,l,N}(\omega)]
\label{align-exp-nabla-minus-RN}
\end{align}
for $m\ge m_0=\cneun+1$.
Note that the middle term in (\ref{align-exp-nabla-minus-RN}) equals
up to a constant 
$\sum_{\Delta\in\T_N}E_{P_{\beta,m,N}}[|\nabla\hat\omega(\Delta)-l\id|^2]$ 
because $\lambda(\Delta)$ is constant for $\Delta\in\T_N$. 
The claim follows from Lemma \ref{lemma-est-expect-part2}. 
\end{proof}

\medskip\noindent\begin{proof}[Proof of Theorem \ref{main-thm}]
For any equilateral triangle with side length 1 having corner points 
$A_1,A_2,A_3\in\R^2$, the map 
\begin{align}
\R^{2\times 2}\ni M \mapsto 
\max\{ \| M(A_2-A_1)\|, \| M(A_3-A_2)\|, \| M(A_1-A_3)\|\} 
\end{align}
is a matrix norm and hence equivalent to any other matrix norm on 
$\R^{2\times 2}$. Thus Theorem \ref{main-thm} follows from Theorem 
\ref{variant-main-thm}.
\end{proof}

\paragraph{Acknowledgements.} 
The research of MH is supported by the Netherlands Organization for Scientific Research (NWO). 

\bibliographystyle{alpha}


\end{document}